\documentclass[12pt,a4paper,
fleqn,reqno]{amsart}                 
\usepackage[T1]{fontenc}                   
\usepackage[utf8]{inputenc}                 
\usepackage[english]{babel}    
 \usepackage{csquotes}
\usepackage{graphicx}                      %
\usepackage[font=small]{quoting}    
\usepackage{float}
\usepackage{caption}  

\usepackage[top=1.4in,bottom=1.8in,left=1.0in,right=1.0in]{geometry}                  
\usepackage{verbatim}
\usepackage{color}
\usepackage{xcolor}
\usepackage{picture}
\usepackage{amsmath,amsthm,amsfonts,amssymb}
\usepackage{comment}
\usepackage{hyperref}
\hypersetup{colorlinks,linkcolor={black}, citecolor={blue},urlcolor={black}}   
\usepackage{enumitem}
\usepackage{mathtools}
\usepackage{wasysym}
\usepackage[url=false,doi=false,isbn=false, giveninits=true,
style=
numeric,
maxbibnames=99]{biblatex}
\newtheorem{thm}{Theorem}[section] 
\newtheorem{lem}[thm]{Lemma} 
\newtheorem{cor}[thm]{Corollary} 
\newtheorem{prop}[thm]{Proposition}

\newtheorem{defn}{Definition}[section]

\theoremstyle{definition} 
\newtheorem{rem}{Remark}[section]

\theoremstyle{definition}

\usepackage{esint}

\usepackage{amsmath}

\def\n{\nabla}

\def\a{\alpha}

\def\n{\nabla}

\def\a{\alpha}

\def\l{\lambda}
\def\s{\sigma}

\def\n{\nabla}
\def\<{\langle}
\def\>{\rangle}

\def\n{\nabla}

\def\RR{\mathbb{R}}

\def\a{\alpha}

\def\l{\lambda}

\def\s{\sigma}

{\left\{\begin{array}{@{}l@{}}}{\end{array}\right.}
\patchcmd{\abstract}{\scshape\abstractname}{\textbf{\abstractname}}{}{}
\makeatletter 
\def\@makefnmark{} 
\makeatother

\numberwithin{equation}{section}
\numberwithin{exa}{section}

\begin{document}
\title[Interior hessian estimate]{Interior $C^{2}$ estimate for semi-convex solutions to a class of Hessian quotient equations in arbitrary dimensions}
 
\author[X. Mei]{Xinqun Mei}
\address[X. Mei]{Key Laboratory of Pure and Applied Mathematics, School of Mathematical Sciences, Peking University,  Beijing, 100871, P.R. China}
\email{\href{qunmath@pku.edu.cn}{qunmath@pku.edu.cn}}
	
\author[J. Yan]{Jin Yan}
\address[J. Yan]{School of Mathematical Sciences, University of Science and Technology of China, Hefei,  Anhui Province,230026, P.R. China}
\email{\href{yjoracle@mail.ustc.edu.cn}{yjoracle@mail.ustc.edu.cn}}



\begin{abstract}
In this paper, we study the interior $C^{2}$ estimates for Hessian quotient equations $\frac{\sigma_{3}(D^{2}u)}{\sigma_{l}(D^{2}u)}=1$ for $l=1, 2$, in arbitrary dimensions, under the natural ellipticity and semi-convexity conditions. We further derive analogous results  for the corresponding sum Hessian equations. In addition, we establish several rigidity results.
  
\end{abstract}

\maketitle


\section{Introduction}
In this paper, we study  interior $C^{2}$ estimates  for Hessian quotient equation
\begin{eqnarray}\label{quo-eq}
\frac{\sigma_{3}(D^{2}u)}{\sigma_{l}(D^{2}u)} =f, \quad \text{in } B_5\in\mathbb R^n,
\end{eqnarray}
for $l=1, 2$. For $\lambda\in \RR^{n}$, let $\sigma_{l}(\lambda)$ denote the $l$-th elementary symmetric function defined by 
\begin{eqnarray*}
\s_{l}(\lambda)\coloneqq  \sum _{1 \le i_1 < i_2 <\cdots<i_l\leq n}\lambda_{i_1}\lambda_{i_2}\cdots\lambda_{i_l}.
\end{eqnarray*}
Let $\lambda(D^{2}u)$ denote the eigenvalues of $D^{2}u$,  the $l$-Hessian operator of $u$ is defined by $\sigma_{l}(D^{2}u)\coloneqq\sigma_{l}(\lambda(D^2{u}))$.  Recall that the  G{\aa}rding cone is given by
\begin{eqnarray*} 
\Gamma_l \coloneqq \left\{ \lambda  \in \mathbb{R}^n \mid \sigma _i (\lambda ) > 0,~~\forall 1 \le i \le l \right\}.
\end{eqnarray*} 
We say a function $u$ is an admissible solution for Eq. \eqref{quo-eq} if $\lambda(D^{2}u)\in \Gamma_{3}$. We call $u$  semi-convex if there exists a constant $K >0$ such that $D^{2}u>-KI$,
where $I$ is the identity matrix. When $l=1,2$, we  establish the interior $C^{2}$ estimates for admissible, semi-convex solutions to Eq. \eqref{quo-eq}.
\begin{thm}\label{thm-quotient}
Let $l=1,2$ and $n\geq3$. Let $u\in C^4(B_5)$ be an admissible, semi-convex solution of
\begin{eqnarray}\label{3-quo-eq}
\frac{\sigma_3(D^2u)}{\sigma_l(D^2u)}=1,\quad \text{in } B_5\in\mathbb R^n.
\end{eqnarray}
Then 
\begin{eqnarray}\label{0-est-quotient}
|D^2u(0)|\leq C,
\end{eqnarray}   
where $C$ depends only on $n$, $K$, and $\|u\|_{C^{1}(B_5)}$.    
\end{thm}
As a direct consequence of the above theorem, we can derive some rigidity results for Hessian quotient equation. Before that, we introduce the definition of sub(super)-quadratic growth.
\begin{defn}
    we say a function $u: \RR^{n}\rightarrow \RR$ has sub-quadratic growth if there exist some positive constants $b, c$ and sufficiently $R$, such that
    \begin{eqnarray*}
        |u(x)|\leq b|x|^{2}+c,\quad \text{for}\quad |x|>R.
    \end{eqnarray*}
    We say $u$ has super-quadratic growth, if $u$ satisfies 
    \begin{eqnarray*}
        u(x)\geq b|x|^{2}-c,\quad \text{for}\quad |x|>R.
    \end{eqnarray*}
\end{defn}
\begin{thm}\label{thm-quotient-rigidity}
    Let $n\geq3$ and $l=1,2$, and let $u\in C^4(\mathbb R^n)$ be an entire admissible, semi-convex solution with sub-quadratic growth satisfying
    \begin{eqnarray*}
       \frac{\sigma_{3}(D^{2}u)}{\sigma_{l}(D^{2}u)}=1, \quad \text{in}\quad \mathbb{R}^{n}. 
    \end{eqnarray*}
    Then $u$ must be a quadratic polynomial.
    \end{thm}

Before our work, some results in the low dimensions were known. When $f=1$ and $l=1$, using the special Lagrangian structure of Eq. \eqref{quo-eq}, Chen--Warren--Yuan \cite{cmy} and Wang--Yuan \cite{wy} established \eqref{0-est-quotient} when $n=3$ and $n=4$, respectively. For nonconstant function $f$, Zhou \cite{zxc} established the interior $C^{2}$ estimates when $n=3,4$ by studying the twisted special Lagrangian equation and Lu \cite{lu-arxiv} used the Legendre transform and obtained similar results when $n=3$. For general Hessian quotient equations in high dimension, there is no special Lagrangian structure. It is therefore of great interest
to know whether interior $C^{2}$ estimates hold.  Recently, Lu \cite{L2025APDE} studied the interior $C^{2}$ estimates for a class of Hessian equations in general dimension. Precisely, for the Hessian quotient equation 
\begin{eqnarray}\label{hessian-quotient-equ}
    \frac{\sigma_{k}(D^{2}u)}{\sigma_{l}(D^{2}u)}=f,
\end{eqnarray}
where $1\leq l<k\leq n$. Lu \cite{L2025APDE} established the interior $C^{2}$ estimates for Eq. \eqref{hessian-quotient-equ} when $k=n$ and $l=n-1, n-2$. Moreover, he constructed a singular solution to Eq. \eqref{hessian-quotient-equ} with $k-l\geq 3$, which implies that  interior $C^{2}$ estimates fail when $k-l\geq 3$. Consequently, the remaining open question is 

\
	
	{\it 
	When $1\leq l<k\leq n$ and $l=k-1, k-2$, whether the interior $C^{2}$ estimates still hold for the solution to Eq. \eqref{hessian-quotient-equ} in arbitrary dimension ?   }

\

When $k=2, l=1$, Lu--Sroka \cite{ls} settled the case $f=1$ by rewriting the Hessian quotient operator $\frac{\s_{2}}{\sigma_{1}}$ as the $\s_{2}$ operator and using the results of Warren--Yuan \cite{wy-cpam} and Shankar--Yuan \cite{SY2020CVPDE,SY2025Annals}. Jiao-Sui \cite{jz} studied the case $f=f(x,u)$. Our results, namely, Theorem \ref{thm-quotient},  give an affirmative answer to this question for semi-convex solutions when $k=3$ and $l=1,2$.

Compared with the Hessian quotient operator, the  results on the study of  interior $C^{2}$ estimates for the Hessian equation are much more abundant. For the Hessian equation 
 \begin{eqnarray}\label{k-equ}
     \s_{k}(D^{2}u)=f,
 \end{eqnarray}
Heinz \cite{he} first established the interior $C^{2}$ estimates for $k=2$ and $n=2$, i.e., the Monge-Amp\`ere equation $\det(D^{2}u)=f$ in dimension $2$. See also recent proofs by Chen--Han--Ou \cite{cho} and Liu \cite{ljk}. 
When  $n\geq 3$ and $k\geq 3$, there are famous singular solutions constructed by Pogorelov \cite{pogo}, later generalized by Urbas \cite{urbas}, which implies the $C^{2}$ interior estimates fail. After that, whether the interior $C^{2}$ estimates hold for Eq. \eqref{k-equ} with $k=2$ is a longstanding problem. A major breakthrough was made by Warren--Yuan \cite{wy-cpam} and Shankar--Yuan \cite{SY2025Annals}, who showed the $C^{2}$ estimate for admissible solution to the $2$-Hessian equation $\s_{2}(D^{2}u)=1$ holds in $n=3$ and $n=4$, respectively. For general $n (n\geq 5)$,  the interior $C^{2}$ estimate was obtained under certain convex conditions on the solution, see, e.g., \cite{cjz, msy, mooney, SY2020CVPDE,  SY2025Annals}.   For general right-hand side $f$, Qiu \cite{qiu} and Fan \cite{fan} solved the problem for $n=3$ and $n=4$, respectively. For general dimension, Guan--Qiu \cite{gq} settled this problem under the condition $\sigma_{3}(D^{2}u)>-A$ for some positive constant $A$. For the interior $C^{2}$ estimates of curvature equation, we  refer the reader to \cite{qiu-curvature, qiu-zhou, suw2004, zxc-curvature} and  the references therein..
 
To prove Theorem \ref{thm-quotient}, a key ingredient is to establish a concavity inequality for the Hessian quotient operator $\frac{\sigma_{k}(D^{2}u)}{\sigma_{l}(D^{2}u)}$. In fact, this concavity inequality plays an important role in deriving various $C^{2}$ estimates for fully nonlinear elliptic equations. When $k=n$, the concavity inequality was established by Guan--Sroka \cite{gs}, see also Zhang \cite{Z2025AIM} for related results on the Hessian operator $\sigma_{k}(D^{2}u)$.  To the best of our knowledge, when $k\leq n-1$, the concavity inequality for Hessian quotient operator remains unknown.
Therefore, establishing the interior $C^{2}$ estimates via a direct analysis of Eq. \eqref{quo-eq} presents significant difficulties. To overcome this issue, we proceed to investigate the interior $C^{2}$ estimates for the following sum Hessian equation 
\begin{eqnarray}\label{sum-hessian}
   S_{k}(D^{2}u)\coloneqq \sigma_{k}(D^{2}u)+\sigma_{k-1}(D^{2}u)=f(x, u, Du).
\end{eqnarray}
Denote 
\begin{eqnarray*}
\tilde\Gamma_k=\Gamma_{k-1}\cap\{\lambda|\sigma_k(\lambda)+\sigma_{k-1}(\lambda)>0\}.
\end{eqnarray*}
 Li-Ren-Wang \cite{LRW2019CVPDE} proved that when $\lambda(D^{2} u)\in\tilde{\Gamma}_{k}$, the operator $S_{k}(D^{2}u)$ is elliptic. In this case,  we say a function $u$ is an admissible solution for Eq. \eqref{sum-hessian}. 
When $k=3$, we establish  interior $C^{2}$ estimates for admissible, semi-convex solution to Eq. \eqref{sum-hessian},  stated as follows. 
\begin{thm}\label{thm-main}
Let $n\geq3$ and let $f\in C^{1,1}(B_5\times\mathbb R\times \mathbb R^n)$ be a positive function. Let $u\in C^4(B_5)$ be an admissible, semi-convex solution of equation
\begin{eqnarray}\label{main-eq}
\sigma_{3}(D^{2}u)+\sigma_{2}(D^{2} u)=f(x, u, Du),\quad \text{in } B_5\in\mathbb R^n.
\end{eqnarray}
Then we have
\begin{eqnarray}\label{0-est}
|D^2u(0)|\leq C,
\end{eqnarray}   
where $C$ depends only on $n$, $K$, $\|u\|_{C^{1}(B_5)}$, $\min\limits_{B_5\times\mathbb R\times\mathbb R^n}f$ and $\|f\|_{C^{1,1}(B_5\times\mathbb R\times\mathbb R^n)}$.
\end{thm}
\begin{rem}\label{rmk-quotient}
By a rescaling argument, Theorem \ref{thm-main} also holds for the equation \begin{eqnarray}\label{a-3-quo-equ}
    \sigma_{3}(D^2u)+\alpha\sigma_2(D^2u)=f(x,u,Du),
    \end{eqnarray}
where $\alpha$ is a positive constant.  
\end{rem}
\begin{rem}
For the equation $$\s_{2}(D^{2}u)+ \a \sigma_{1}(D^{2}u)=f,$$ by introducing the transformation $\tilde{v}\coloneqq u+\frac{a}{2}|x|^2$ with $a=\frac{\alpha}{n-1}$ and applying the results in \cite{fan,  qiu, SY2025Annals}, we can show that \eqref{0-est}  holds  for admissible solutions when $n=2,3,4$ and also for $n\geq 5$ under the condition 
\begin{eqnarray*}
\lambda_{min}(D^2u)+ \frac{\alpha}{n-1}\geq -c(n)\left(\Delta u+\frac{n\alpha}{n-1}\right),
\end{eqnarray*}
where $c(n)\coloneqq \frac{\sqrt{3n^2+1}-n+1}{2n}.$
\end{rem}

A key observation in this paper is that, for any admissible, semi-convex solution $u$ to Eq. \eqref{3-quo-eq}, if we consider the function 
$$v\coloneqq u-\frac{a}{2}|x|^{2},$$
for some positive constant $a$, then $v$ is an admissible, semi-convex solution to Eq. \eqref{a-3-quo-equ} with $f\equiv1$, see Eq. \eqref{eq-quotient-v} and 
Eq. \eqref{equ-v2}. It then follows from Theorem \ref{thm-main} that Theorem \ref{thm-quotient} holds.

For $k\geq 4$, we construct a singular solution to Eq. \eqref{sum-hessian}, which shows that interior $C^{2}$ estimates fail for Eq. \eqref{sum-hessian} as $k\geq 4$.

\begin{thm}\label{thm-counter example}
    Let $n\geq 4$ and $k\geq 4$. Then there exists a convex viscosity solution $u$ of 
    \begin{eqnarray*}
        \s_{k}(D^{2}u)+\s_{k-1}(D^{2}u)=f(x),\quad \text{in}\quad B_{r}\subset \RR^{n}, 
        \end{eqnarray*}
        for some constant $r>0$ and smooth function $f$ on $B_{r}$, such that $u\in C^{0,1}(B_{r})$, but $u\notin C^{1,\beta}(B_{\frac{r}{2}})$ for some $\beta>0$.
\end{thm}

\begin{rem}
In view of the above results, the remaining question concerning   interior $C^{2}$ estimates for Hessian sum equation \eqref{sum-hessian} is whether such estimates hold for admissible solutions to Eq. \eqref{main-eq} in arbitrary dimensions.
\end{rem}

An important application of the interior $C^2$ estimates is the derivation of rigidity theorems for certain equations.  Consider the following $k$-Hessian equation in $n$-dimensional Euclidean spaces:
\begin{eqnarray}\label{eq-sigmak=1}
    \sigma_k(D^2u)=1,\quad \text{in} \quad\mathbb R^n.
\end{eqnarray}
For $k=1$, Eq. \eqref{eq-sigmak=1} is linear and the result follows directly from the Liouville property of harmonic functions. For $k=n$, Eq. \eqref{eq-sigmak=1} reduces to the Monge-Ampère equation, for which a well-known result in geometry asserts that that any entire strictly convex solution to this equation must be a quadratic polynomial. This was established by Jörgens \cite{J1954MA}, Calabi \cite{C1958MMJ}, and Pogorelov \cite{P1972GD,pogo}. Cheng-Yau \cite{CY1986CPAM} later provided an alternative geometric proof, and Caffarelli-Li \cite{CL2003CPAM} further extended the classical result of Jörgens, Calabi, and Pogorelov theorem based on the theorey of Monge-Amp\`ere equations \cite{c-2,c-2}. For $k=2$, Chang-Yuan \cite{CY2010DCDS} proved that if the Hessian satisfies 
$D^2u\geq [\delta-\sqrt{\frac{2n}{n-1}}] I$,
for some $\delta>0$, then any entire solution of  Eq. \eqref{eq-sigmak=1} must be a quadratic polynomial. They conjectured that any semi-convex entire solution of \eqref{eq-sigmak=1}  must be quadratic polynomials. This conjecture was later confirmed by 
 Shanker-Yuan \cite{SY2022Duke}.

For general $3\leq k\leq n-1$, Bao-Chen-Guan-Ji \cite{BCGM2003AJM} showed that any entire strictly convex  solution of Eq. \eqref{eq-sigmak=1} with sup-quadratic growth are quadratic polynomials. Later, Li-Ren-Wang \cite{LRW2016JFA} extended this result by relaxing the strictly convex condition to $(k+1)$-convex solutions. More recently, Zhang \cite{Z2025AIM} further extended this result to the class of $k$-convex and semi-convex solutions. All of the above  results rely on Pogorelov-type $C^2$ interior estimate and, in an essential way, require sup-quadratic growth condition. Without such a condition, to the best of our knowledge, there are no corresponding results for general $3\leq k\leq n-1$. Moreover,  Warren \cite{W2016CPDE} constructed examples showing that \eqref{eq-sigmak=1} admits non-polynomial entire $k$-convex solutions to Eq. \eqref{eq-sigmak=1} when $n\geq 2k-1$.

For the sum Hessian equation 
\begin{eqnarray*}
\sigma_k(D^2u)+\alpha\sigma_{k-1}(D^2u)=1,\quad \text{in} \quad \mathbb R^n.
\end{eqnarray*}
 Liu-Ren \cite{LR2023JFA} first obtained the result for $k$-convex solutions with sup-quadratic growth. Later, Liang--Yan--Zhu \cite{LYZ2024Arxiv} extended the result to $(k-1)$-convex and semi-convex solutions. Recently, Dong--Xu--Zhang \cite{DXZ2026arxiv} obtained the result under a dynamic semi-convex condition.  The sup-quadratic growth condition is essential in all these works.

In the final part of this paper, we establish two rigidity results for the equation
\begin{eqnarray}\label{eq-rigidity}
\sigma_{3}(D^2u)+\sigma_2(D^2u)=c_0,\quad \text{in } \mathbb R^n,  
\end{eqnarray}
where $c_0>0$ is a positive constant.

\begin{thm}\label{thm-rigidity}
Let $n\geq3$, and let $u\in C^4(\mathbb R^n)$ be an admissible solution of Eq. \eqref{eq-rigidity}. Assume that one of the following conditions holds: 
\begin{enumerate}
\item $u$ is semi-convex and has sub-quadratic growth, 
\item $u$ satisfies $D^2u\geq \left(-\frac{1}{n-2}+\delta\right)I$ for some positive constant $\delta>0$, and $c_0> -\frac{n(n-1)(n-2)}{6}\left(\frac{1}{n-2}-\delta\right)^3+\frac{n(n-1)}{2}\left(\frac{1}{n-2}-\delta\right)^2$,
\item $c_0\leq\frac{2n(n-1)}{3(n-2)^2}$ and $u$ satisfies $D^2u\geq\left(-K_0+\delta\right)I$ for the unique $K_0\in(0,\frac{2}{n-2}]$ such that $-\frac{n(n-1)(n-2)}{6}K_0^3+\frac{n(n-1)}{2}K_0^2=c_0$ and some $\delta>0$,
\end{enumerate}
then $u$ must be a quadratic polynomial.
\end{thm}
As a direct consequence of Theorem \ref{thm-rigidity} (2), we obtain the following rigidity result for convex solutions to Eq. \eqref{eq-rigidity} without the quadratic growth assumption.
\begin{cor}
    Let $n\geq3$, and let $u\in C^4(\mathbb R^n)$ be  a convex solution of Eq. \eqref{eq-rigidity}. Then $u$ must be  a quadratic polynomial. 
\end{cor}



We conclude the introduction by outlining the main ideas of the proof of Theorem \ref{thm-main}. We follow the strategy presented in \cite{SY2020CVPDE} and \cite{L2025APDE}, we divide the proof into the following three steps. The first step is to establish a Jacobi-type inequality for $b=\log\lambda_{\max}(D^{2}u)$. The second step is to bound the value of $b(0)$ in terms of its integral by employing the Legendre transform together with a mean-value-type inequality. The third step is to handle the integral terms via integration by parts.

The strategy works well for the $\sigma_2$ equation, but encounters difficulties at each step for the $S_3$ equation. In the first step, due to the lack of homogeneity of $S_3$, we cannot derive the concavity inequality by the method in \cite{Z2025AIM}.  However,  a  key observation of Dong-Xu-Zhang \cite{DXZ2026arxiv} shows that the desired concavity inequality can be obtained from a higher-dimensional concavity inequality for the $\sigma_3$ operator. The main difficulty in the second step lies in ensuring the uniform ellipticity of a elliptic operator $G^{ij}$ (see Section~\ref{sec-4}) through a suitable normalization.  For the $\sigma_2$ equation, this is a straightforward consequence  since $\lambda_2$ is bounded. In contrast, for $S_3$ equation,  we only have a bound on $\lambda_3$, which makes it difficult to estimate the term $G^{22}$. A substantially more delicate argument is therefore required to establish  Lemma \ref{lem-uniform-elliptic}. 
In the third step, unlike the $\sigma_2$ case, we cannot control $\sigma_2(D^2u)$ solely by $\Delta u$, consequently, the $\sigma_2$ term  appearing  in the integral must be handled with caution. Moreover, since $u$ only belongs to $\tilde\Gamma_3$ (rather than $\Gamma_3$), the quantities $\sigma_2^{ii}$ and $F^{ii}$ are not comparable. Motivated by the idea in \cite{L2025APDE}, we  overcome this difficulty by establishing the key Lemma \ref{lem-crucial-ineq}, which completes the proof.

\

{\textbf{Organization of the rest of the paper.}}
In Section~\ref{sec-2}, we review the basic properties of $\sigma_k$ and $S_k$ and collect some known facts. In Section~\ref{sec-3}, we establish a Jacobi inequality for $b=\log\lambda_{\max}$. In Section~\ref{sec-4}, we perform a Legendre transform and prove the uniform ellipticity of $G^{ij}$ under a suitable normalization.  In Section~\ref{sec-5}, we prove a key lemma showing the terms $\sigma_2^{ii}$ and $F^{ii}$ are comparable. In Section~\ref{sec-6}, we complete the proof of Theorem~\ref{thm-main} and Theorem~\ref{thm-quotient}. Section~\ref{sec-7} is devoted to constructing a singular solution for $S_k$ equations with $k\geq 4$. The final section establishes some rigidity results.

\section{Preliminaries}\label{sec-2}

In this section, we first recall the definition and some basic properties of the elementary symmetric polynomial functions.  
\begin{defn}
For $k = 1, 2,\ldots, n,$ the $k$-th elementary symmetric function $\sigma_k$ is defined by
\begin{eqnarray*} 
\sigma_k(\lambda) = \sum _{1 \le i_1 < i_2 <\cdots<i_k\leq n}\lambda_{i_1}\lambda_{i_2}\cdots\lambda_{i_k},
 \qquad \text {for} \quad\lambda\coloneqq (\lambda_1,\ldots,\lambda_n)\in \mathbb{R}^{n}.
\end{eqnarray*}
\end{defn}
We adopt the convention that $\sigma_0=1$ and $\sigma_k =0$ for $k>n$.  Denote  $\sigma_{k;i}(\lambda)\coloneqq \sigma _k (\lambda \left| i \right.)$ the symmetric
	function with $\lambda_i = 0$ and $\sigma_{k;ij}(\lambda)\coloneqq \sigma _k (\lambda \left| ij \right.)$ the symmetric function with $\lambda_i =\lambda_j = 0$.  
\begin{defn}
    Let $A=\{A_{ij}\}$ be an $n\times n$ symmetric matrix. For $k=1,2,\cdots, n$, we define
    \begin{eqnarray*}
        \sigma_{k}(A)\coloneqq \sigma_{k}(\lambda(A))=\sum\limits_{1\leq i_{1}<i_{2}\cdots< i_{k}\leq n}\lambda_{i_{1}}(A)\lambda_{i_{2}}(A)\cdots \lambda_{i_{k}}(A),
    \end{eqnarray*}
where $\lambda(A)=(\lambda_1(A),\lambda_2(A),\ldots,\lambda_n(A))$ denotes the vector of eigenvalues of $A$. 
Equivalently, $\sigma_k(A)$ is the sum of all $k\times k$ principal minors of $A$.
\end{defn}
We also denote by $\sigma _m (A \left|
i \right.)$ the symmetric function with $A$ deleting the $i$-row and
$i$-column and $\sigma _m (A \left| ij \right.)$ the symmetric
function with $A$ deleting the $i,j$-rows and $i,j$-columns. Define 
\begin{eqnarray*}
    S_{k}(\l)\coloneqq \sigma_{k}(\lambda)+\sigma_{k-1}(\l), \quad \text{for}~k\geq 1,
\end{eqnarray*}
and 
\begin{eqnarray*}
    F(D^{2}u)\coloneqq S_{k}(\lambda(D^{2} u)).
\end{eqnarray*}

Next, we recall some well-known properties of elementary functions.
\begin{prop}
Let $\lambda=(\lambda_1,\cdots,\lambda_n)\in \mathbb R^n$. For any $1\leq p,q \leq n$, we denote
\begin{enumerate}
\item[$(4)$] $S_k^{pp}(\lambda):=\frac{\partial S_k(\lambda)}{\partial \lambda_p}=\sigma_{k-1;p}(\lambda)+\sigma_{k-2;p}(\lambda)$,
\item[$(5)$] $S_k^{pp,qq}(\lambda):=\frac{\partial S_k(\lambda)}{\partial \lambda_p \partial \lambda_q}=\sigma_{k-2;pq}(\lambda)+\sigma_{k-3;pq}(\lambda)$, \quad \text{and} \quad $S_k^{pp,pp}(\lambda)=0$,
\item[$(6)$] $S_k(\lambda)=\lambda_p S_{k-1;p}(\lambda)+S_{k;p}(\lambda)$,
\item[$(7)$] $\sum\limits_{p=1}^n S_{k;p}(\lambda)=(n-k)\sigma_{k}(\lambda)+(n-k+1)\sigma_{k-1}(\lambda)$, 
\item[$(8)$] $\sum\limits_{p=1}^{n}\lambda_p S_{k-1;p}(\lambda)=k\sigma_k(\lambda)+(k-1)\sigma_{k-1}(\lambda)$.
\end{enumerate}
\end{prop}

\begin{lem}\label{lem-2.1}
Assume that $\lambda=(\lambda_1,\cdots,\lambda_n)\in\tilde\Gamma_k$, $1\leq k\leq n$, and $\lambda_1\geq\cdots\geq\lambda_n$. Then
\begin{enumerate}
    \item  The set $\tilde{\Gamma}_{k}$ is a convex set. Moreover, the operator $S_k$ is elliptic in $\widetilde{\Gamma}_{k}$.  

\item The functions $S_{k}^{\frac{1}{k}}(\lambda)$ and $\log S_{k}(\lambda)$ are concave in $\tilde\Gamma_k$.


\item  For any $1\leq j\leq k-1$, there exists a positive constant $c_1$, depending on $n,k$, such that
\begin{eqnarray*}
F^{jj}(\lambda)\geq \frac{c_1 F(\lambda)}{\lambda_j}.
\end{eqnarray*}

\item For any $j\geq k$, there exists a positive constant $c_2$, depending on $n,k$, such that
\begin{eqnarray*}
F^{jj}(\lambda)\geq c_2 \sum_{i}F^{ii}(\lambda).
\end{eqnarray*}
\item If $S_k(\lambda)\leq N_1$ and $\lambda_i\geq-K$ for all $1\leq i\leq n$, then there exists a positive constant $C_3$, depending on $N_1,n,k,K$, such that $|\lambda_i|\leq C_3$ for all $i\geq k$.
\end{enumerate}
\end{lem}
\begin{proof}
    The detailed proof can be found in \cite[Theorem 10, Corollary 13]{LRW2019CVPDE} for (1), (2), and \cite[Lemma 2.4 (b)]{LR2023JFA} for (3), and \cite[Proposition 1.6 (ii)]{R2024IMRN} for (4), and \cite[Lemma2.5 (b)]{LYZ2024Arxiv} for (5). 
\end{proof}

Denote
\begin{eqnarray*}
    F^{ij}(D^{2}u)\coloneqq \frac{\partial F(D^{2}u)}{\partial u_{ij}} ,\quad F^{ij,pq}=\frac{\partial F(D^{2}u)}{\partial u_{ij}\partial u_{pq}}.
\end{eqnarray*}

Finally, we recall an important concavity inequality related to the sum Hessian operator $F(D^{2}u)$.
\begin{lem}\label{lem-concave-ineq}
Suppose $\lambda_1 \ge \lambda_2 \ge \cdots \ge \lambda_n \ge -K$. Then for sufficiently large $\lambda_1$ depending only on $n$, the following inequality holds:
\begin{eqnarray}\label{concavity}
    -\sum_{p \neq q} \frac{F^{pp,qq}\xi_p\xi_q}{F} + 2 \frac{\left( {\sum_{i=1}^n F^{ii}\xi_i} \right)^2}{F^2} + 2 \sum_{i\geq 2} \frac{F^{ii}\xi_i^2}{(\lambda_1+K+1) F} \ge (1+\gamma) \frac{F^{11}\xi_1^2}{\lambda_1 F},
\end{eqnarray}
for any vector $\xi = (\xi_1, \ldots, \xi_n) \in \mathbb{R}^n$ and some constant $\gamma>0$.
\end{lem}

\begin{proof}
    We  refer to \cite[Lemma 3.2]{DXZ2026arxiv} for the detailed proof. 
\end{proof}

\section{Jacobi inequality}\label{sec-3}

In this section, we first derive a Jacobi-type inequality using the concavity inequality \eqref{concavity}. We then establish an integral version of the Jacobi inequality.  
\begin{lem}\label{lem-jacobi-ineq}
Let $n\geq3$. Let $ f \in C^{1,1}(B_{5} \times \mathbb{R}\times\mathbb R^n) $ be a positive function and let $ u \in C^4(B_{5}) $ be a semi-convex admissible solution of Eq. \eqref{main-eq}. For any $ x_0 \in B_4 $, suppose that $ D^2u $ is diagonalized at $ x_0 $, assume that $ \lambda_1 \geq \cdots \geq \lambda_n\geq-K $ and $\lambda_1\geq J(n,K)$  at $x_{0}$, where $J(n, K)$ is a sufficiently large constant. Let $b=\log\lambda_1$. Then, at $x_0$, we have
\begin{eqnarray*}
\sum_i F^{ii}b_{ii} \geq \varepsilon_0 \sum_i F^{ii}b_i^2-C(\Delta u+1)
\end{eqnarray*}
holds in the viscosity sense, where $\varepsilon_0>0$  depends only on $n$, and $C$ depends on $n$, $K$, $\|u\|_{C^{1}({B}_5)}$, $\min\limits_{{B}_5\times\mathbb R\times\mathbb R^n}f$ and $\|f\|_{C^{1,1}({B}_5 \times\mathbb R\times\mathbb R^n)}$.  
\end{lem}
\begin{proof}
At $x_{0}$, assume that $\lambda_1=\cdots=\lambda_{m}$ for some positive integer $m$. By \cite[Lemma 5]{BCD2017ACTA}, we have
\begin{eqnarray}\label{one-deri}
\quad \delta_{\alpha\beta} \cdot (\lambda_{1})_i = u_{\alpha\beta i}, \quad 1 \leq \alpha, \beta \leq m,
\end{eqnarray}
and 
\begin{eqnarray*}
(\lambda_{1})_{ii} \geq u_{11ii} + 2 \sum_{p > m} \frac{u_{1pi}^2}{\lambda_1 - \lambda_p}
\end{eqnarray*}
in the viscosity sense. 

At $x_{0}$, suppose that $u_{11}\geq u_{22}\geq \cdots \geq u_{nn}$ and since $b=\log \lambda_{1}$, it follows that
\begin{eqnarray}\label{jacobi-ineq-bi}
b_i=\frac{(\lambda_{1})_{i}}{\lambda_{1}}=\frac{u_{11i}}{\lambda_1 },
\end{eqnarray}
and
\begin{eqnarray*}
\quad b_{ii}=\frac{(\lambda_{1})_{ii}}{\lambda_{1}}-\frac{(\lambda_{1})_{i}^{2}}{\lambda_{1}^{2}}\geq\frac{u_{11ii}}{\lambda_1 }+\frac{2\sum\limits_{p>m}u_{1pi}^2}{\lambda_1 (\lambda_1-\lambda_p)}-\left(\frac{u_{11i}}{\lambda_1 }\right)^2.
\end{eqnarray*}
Therefore, we get
\begin{eqnarray}
F^{ii}b_{ii}-2\varepsilon_0 F^{ii}b_i^2&\geq &\lambda_1 ^{-1}\left(F^{ii}u_{ii11}+2\sum_{p>m}\frac{F^{ii}u_{1pi}^2}{\lambda_1-\lambda_p} \right)-(1+2\varepsilon_0)F^{ii}\frac{u_{11i}^2}{\lambda_{1}^{2}}\notag \\
&\geq &\lambda_1 ^{-1}\left(F^{ii}u_{ii11}+2\sum_{p>m}\frac{F^{11}u_{11p}^2}{\lambda_1-\lambda_p}+2\sum_{p>m}\frac{F^{pp}u_{1pp}^2}{\lambda_1-\lambda_p} \right)-(1+2\varepsilon_0)F^{ii}\frac{u_{11i}^2}{\lambda_{1}^{2}}.~~~~~\quad ~\label{key-1}
\end{eqnarray}
Differentiating Eq. \eqref{main-eq} on both sides and using \cite[Lemma 2.3]{suw2004}, we obtain
\begin{eqnarray}\label{f-one-deri}
    F^{ii}u_{ii1}=f_{x_{1}}+f_{u}u_{1}+\sum_kf_{u_k}u_{k1},
\end{eqnarray}
and
\begin{eqnarray}
    F^{ii}u_{ii11}&=&-F^{ij,rs}u_{ij1}u_{rs1}+(f(x,u,Du))_{11}\notag \\
    &=&-F^{ii,jj}u_{ii1}u_{jj1}+\sum_{i\neq j}\frac{F^{jj}-F^{ii}}{\lambda_i-\lambda_j}u_{ij1}^2+(f(x,u,Du))_{11}. 
    \label{key-2}
\end{eqnarray}
By \cite[Lemma 1.5]{spruck} and the concavity of $\log F(\lambda)$, we derive
\begin{eqnarray*}
    \sum\limits_{i\neq j}\frac{F^{jj}-F^{ii}}{\lambda_{i}-\lambda_{j}}u_{ij1}^{2}&=&2\sum\limits_{i>j>1}\frac{F^{jj}-F^{ii}}{\lambda_{i}-\lambda_{j}}u_{ij1}^{2}+2\sum\limits_{i>1}\frac{F^{11}-F^{ii}}{\lambda_{i}-\lambda_{1}}u_{11i}^{2}\\
    &\geq& 2\sum\limits_{i>1}\frac{F^{11}-F^{ii}}{\lambda_{i}-\lambda_{1}}u_{11i}^{2},
\end{eqnarray*}
together with \eqref{key-2} and \eqref{jacobi-ineq-bi}, we have
\begin{eqnarray}\label{key-5}
    F^{ij}u_{ii11}\geq -F^{ii,jj}u_{ii1}u_{jj1}+2\sum\limits_{i>1} \frac{F^{11}-F^{ii}}{\lambda_{i}-\lambda_{1}}u_{11i}^{2}-C\lambda_1|Db|-C(\Delta u+1).
    \end{eqnarray}
From \eqref{one-deri}, we get
\begin{eqnarray}\label{key-3}
    u_{11i}=u_{1i1}=\delta_{1i}(\lambda)_{1}=0,\quad \forall ~1<i\leq m,
\end{eqnarray}
and similarly we have
\begin{eqnarray}\label{key-4}
    u_{ii1}=0, \quad \forall~ 1<i\leq m. 
\end{eqnarray}
Combining \eqref{key-1}, \eqref{key-5}, \eqref{key-3} and \eqref{key-4}, we have
\begin{eqnarray}
&&F^{ii}b_{ii}-2\varepsilon_0 F^{ii}b_i^2 \notag \\
&\geq&-C+\lambda_1 ^{-1}\left(-F^{ii,jj}u_{ii1}u_{jj1}+2\sum_{p>1}\frac{F^{pp}u_{pp1}^2}{\lambda_1-\lambda_p}-(1+2\varepsilon_0)\frac{F^{11}u_{111}^2}{\lambda_1 } \right) \notag \\
&&+\lambda_1^{-1}\sum_{p>1}\left(\frac{2}{\lambda_1-\lambda_p}-\frac{1+2\varepsilon_0}{\lambda_1}\right)F^{pp}u_{11p}^2-C|Db|. \label{key-6}
\end{eqnarray}
Since $\lambda_{1}\geq J(n, K)$ for some sufficiently large constant $J(n, K)$ and $\lambda_p\geq-K$, we conclude that 
\begin{eqnarray}\label{positive-sum}
\sum_{p>1}\left(\frac{2}{\lambda_1-\lambda_p}-\frac{1+2\varepsilon_0}{\lambda_1}\right)F^{pp}u_{11p}^2>0,
\end{eqnarray}
on the other hand, from \eqref{f-one-deri}, we have
\begin{eqnarray}\label{square-control}
    (\sum\limits_{i=1}^{n}F^{ii}u_{ii1})^{2}\leq C(\Delta u+1)^2,
\end{eqnarray}
and using Lemma \ref{lem-concave-ineq}, we derive
\begin{eqnarray*}
 F^{ii,jj}u_{ii1}u_{jj1}+2\frac{\left(\sum\limits_{i=1}^{n}F^{ii}u_{ii1}\right)^2}{F}+2\sum_{p>1}\frac{F^{pp}u_{1pp}^2}{\lambda_1-\lambda_p}-(1+2\varepsilon_0)\frac{F^{11}u_{111}^2}{\lambda_1 }\geq 0.
 \end{eqnarray*}
 Together with \eqref{key-6}, \eqref{positive-sum} and \eqref{square-control}, we get 
 \begin{eqnarray*}
F^{ii}b_{ii}&\geq& 2\varepsilon_{0}F^{ii}b_{i}^{2}-C|Db|-C(\Delta u+1)\\
&\geq&\varepsilon_{0}F^{ii}b_{i}^{2}-C\sum_i\frac{1}{F^{ii}}-C(\Delta u+1)\\
&\geq&\varepsilon_{0}F^{ii}b_{i}^{2}-C(\Delta u+1).
\end{eqnarray*}
In the last inequality, we  used the fact that $F^{ii}\geq F^{11}\geq C\lambda_1^{-1}$ (see Lemma \ref{lem-2.1} (3)). Thus, we complete the proof of Lemma \ref{lem-jacobi-ineq}.

\end{proof}

As a direct consequence of the above Jacobi inequality, we derive the following integral version Jacobi inequality. 

\begin{lem}
Let $n\geq3$. Let $f\in C^{1,1}(B_{5}\times\mathbb{R}\times\mathbb R^n)$ be a positive function and let $u\in C^4(B_{5})$ be a semi-convex admissible solution of Eq. \eqref{main-eq}. Define
\begin{eqnarray}\label{w-b}
\tilde{b}(x):=\log \max(\lambda_1, J(n,K)), 
\end{eqnarray}
where $J(n,K)$ is the constant in Lemma \ref{lem-jacobi-ineq}. Then, for any non-negative $\varphi\in C_c^{\infty}(B_2)$, we have
\begin{eqnarray}\label{inter-jacobi}
\int_{B_2}\varphi F^{ij}\tilde{b}_i\tilde{b}_j\,dx\leq -C\int_{B_2} F^{ij}\tilde{b}_i\varphi_j\,dx+C,
\end{eqnarray}
where $C$ is a positive constant depending only on $n$.
\end{lem}
\begin{proof}
From the definition of viscosity subsolution, $\tilde{b}(x)$ is a viscosity subsolution of 
\begin{eqnarray*}
F^{ij}\tilde{b}_{ij}\geq\varepsilon_0F^{ij}\tilde{b}_i\tilde{b}_j-C(\Delta u+1).
\end{eqnarray*}
Then, from \cite[Theorem 1]{I1995FE}, we derive that $\tilde{b}$ is also a subsolution in weak sense. That is 
\begin{eqnarray}\label{inte-inequ}
\int_{B_2}\varphi F^{ij}\tilde{b}_i\tilde{b}_j\,dx\leq -C\int_{B_2} F^{ij}\tilde{b}_i\varphi_j\,dx+C\int_{B_2}\varphi\Delta u\,dx+C.
\end{eqnarray}
Moreover, by divergence Theorem, we have
\begin{eqnarray*}
\int_{B_2}\varphi\Delta u\,dx=-\int_{B_2}D\varphi\cdot Du\,dx\leq C.
\end{eqnarray*}
Substituting this estimate into the previous inequality \eqref{inte-inequ} yields \eqref{inter-jacobi}.
\end{proof}

\section{Legendre-Lewy transform}\label{sec-4}
In this section, our main objective is to derive a mean value type inequality which implies the the value of $b(0)$ can be controlled by its integral over $B_{1}$. The key observation is that, after applying the Legendre-Lewy transform, the function
\begin{eqnarray*}
    \phi=\tilde{b}+A|y|^{2}
\end{eqnarray*}
is a subsolution to a uniformly elliptic equation (see \eqref{wide-b}, \eqref{sub-equ}). The idea is inspired by Shankar--Yuan \cite{SY2020CVPDE}.

Without loss of generality, we may assume $u(0)=|Du(0)|=0$. Set $$K_1\coloneqq K+1,$$ and let $w(y)$ be the Legendre transform of function $u+\frac{K_1}{2}|x|^2$, we have 
\begin{eqnarray*}
    (x, Du(x)+K_{1}x)=(D\omega(y), y),
\end{eqnarray*}
and $y(x)=Du(x)+K_{1}x$ is a diffeomorphism.

Define
$$
G\left(D^2w(y)\right)\coloneqq -F\left(-K_1I+(D^2w(y))^{-1}\right)=-F\left(D^2u(x)\right).
$$
Suppose $D^2u$ is diagonal at $x_0$. Then at $x_0$ we have
\begin{eqnarray}\label{g-w}
    G^{ii}=F^{ii}w_{ii}^{-2}=F^{ii}\left(K_1+u_{ii}\right)^2.
\end{eqnarray}
We first prove that $G^{ij}$ is uniformly elliptic by a suitable normalization.
\begin{lem}\label{lem-uniform-elliptic}
Let $n\geq3$ and $f\in C^{1,1}(B_5\times\mathbb R\times \mathbb R^n)$ be a positive function. Let $u\in C^4(B_5)$ be an admissible, semi-convex solution of Eq. \eqref{main-eq}. Then there exists a constant $C>0$ such that  
\begin{eqnarray}\label{elliptic}
C^{-1}\leq\frac{G^{ii}(D^2w)}{\sigma_2(D^2u)+\sigma_1(D^2u)}\leq C.
\end{eqnarray}
\end{lem}
\begin{proof}
From \eqref{g-w}, we have \eqref{elliptic} is equivalent to showing that, for each $i$,
\begin{eqnarray}\label{uniform-elliptic-2}
C^{-1}\leq \frac{F^{ii}(D^2u)\left|K_1+u_{ii}\right|^2}{\sigma_2(D^2u)+\sigma_1(D^2u)}\leq C,\quad \forall~1\leq i\leq n.
\end{eqnarray}
In what follows, we abbreviate $\sigma_{k}(D^{2}u)$ as $\s_{k}$ for $k=1,2$.  We will use the notation $a\sim b$ if the there exists a positive constant $C$ such that $C^{-1}b\leq a\leq Cb$. The constant
$C$ is a uniform positive constant, depending on the   $n$, $K$, $\|u\|_{C^{1}({B}_5)}$, $\min\limits_{{B}_5\times\mathbb R\times\mathbb R^n}f$ and $\|f\|_{C^{1,1}({B}_5 \times\mathbb R\times\mathbb R^n)}$. Its value may change from line to line and we will not track its precise dependence.

We may assume $\lambda_1$ is sufficiently large,  otherwise, $\lambda(D^{2}u)$ lies in a compact set in $\tilde{\Gamma}_{k}$ and  \eqref{elliptic} follows immediately. Next, we divide the proof into two cases.

\

\textit{Case 1} $\lambda_2\leq C_0$ for some sufficiently large $C_0>0$. In this case, since
\begin{eqnarray}\label{lambda}
    \lambda(D^{2}u)\in \tilde{\Gamma}_{3}\subset \Gamma_{2}, \quad \text{and}\quad D^{2}u\geq -KI,
    \end{eqnarray}
we have 
\begin{eqnarray}\label{lambda-1}
    |\lambda_{i}|\leq \max\{K, C_{0}\},\quad \forall ~2\leq i\leq n,
\end{eqnarray}
and this implies
\begin{eqnarray} \label{2+1}
\sigma_2+\sigma_1\sim\lambda_1.
\end{eqnarray}
Next, we \textbf{claim} that 
\begin{eqnarray}\label{case1-claim}
\sigma_{1;12}\geq -1+\varepsilon
\end{eqnarray}
for some sufficiently small constant $\varepsilon=\varepsilon(C_0)>0$. 

By Eq. \eqref{main-eq}, we have
\begin{eqnarray}\label{est-expansion-lambda1}
-C&\leq& \sigma_{3}(\lambda)+\sigma_{2}(\lambda)= \lambda_1\left(\sigma_{1;1}+\sigma_{2;1}\right)+(\sigma_{2;1}+\sigma_{3;1}),
\end{eqnarray}
and \eqref{lambda-1} implies
\begin{eqnarray}\label{F11}
    |\sigma_{2;1}+\sigma_{3;1}|\leq C,
\end{eqnarray}
together this with \eqref{est-expansion-lambda1}, we derive
\begin{eqnarray}\label{est-expansion}
    -C\leq \lambda_{1}(\sigma_{1;1}+\sigma_{2;1})&=&\lambda_1\left(\lambda_2+\sigma_{1;12}+\lambda_2\sigma_{1;12}+\sigma_{2;12}\right).
\end{eqnarray}
The Cauchy inequality implies
\begin{eqnarray}\label{sigma212}
\sigma_{2;12}\leq\frac{n-3}{2n-4}|\sigma_{1;12}|^2.
\end{eqnarray}

If $\sigma_{1;12}<-1$, using \eqref{lambda} again, and  from  \cite[Lemma 2.4]{hg-acta}, we get
\begin{eqnarray}\label{sigma11}
   \sigma_{1;1}=\lambda_2+\sigma_{1;12}>0,\quad \text{and}\quad  \lambda_{2}>0.
\end{eqnarray}
Combining \eqref{sigma11}, \eqref{sigma212} and  \eqref{est-expansion-lambda1}, we derive that
\begin{eqnarray*}
\lambda_{2}(1+\sigma_{1;12})+\sigma_{1;12}+\sigma_{2;12}
&\leq&-\sigma_{1;12}(1+\s_{1;12})+\sigma_{1;12}+\sigma_{2;12} \\
&\leq&-|\sigma_{1;12}|^2+\frac{n-3}{2n-4}|\sigma_{1;12}|^2<-\frac{n-1}{2n-4}.
\end{eqnarray*}
This contradicts \eqref{est-expansion} provided $\lambda_1$ is sufficiently large.

If $-1+\varepsilon\geq \sigma_{1;12}\geq-1$, then we have
\begin{eqnarray*}
\lambda_2(1+\sigma_{1;12})+\sigma_{1;12}+\sigma_{2;12}\leq\varepsilon C_0-1+\varepsilon+\frac{n-3}{2n-4}|1-\varepsilon|^2<-C
\end{eqnarray*}
provided $\varepsilon>0$ sufficiently small, which again leads to a contradiction. Hence the \textbf{claim} holds.

In the following, according to the range of the index $i$, we show \eqref{elliptic} holds one by one.

\begin{enumerate}
    \item  $i=1$. By Lemma \ref{lem-2.1} (3), we have
\begin{eqnarray*}
F^{11}|K_1+u_{11}|^2\sim F^{11}\lambda_1^2\geq c_{1}F(\lambda)\lambda_{1} \geq c\lambda_1.
\end{eqnarray*}
On the other hand, from \eqref{F11}, we get
\begin{eqnarray*}
F^{11}\lambda_1^2=\lambda_1\left(\sigma_3+\sigma_2-\sigma_{3;1}-\sigma_{2;1}\right)\leq C\lambda_1,
\end{eqnarray*}
then  \eqref{uniform-elliptic-2}   for $i=1$ follows from \eqref{2+1}.

\item  $i=2$.  Using \eqref{lambda-1} and \eqref{case1-claim}, we obtain
\begin{eqnarray}\label{case1-i=2-F22}
F^{22}=\sigma_{2;2}+\sigma_{1;2}\geq \lambda_1(1+\sigma_{1;12})-C\geq \varepsilon \lambda_1-C, 
\end{eqnarray}

and 
\begin{eqnarray*}
F^{22}=\sigma_{2;2}+\sigma_{1;2}\leq\lambda_1(1+\sigma_{1;12})+C\leq C\lambda_1.
\end{eqnarray*}
Since
\begin{eqnarray*}
   1 \leq|K_{1}+u_{22}|\leq C,
\end{eqnarray*}
we conclude that \eqref{uniform-elliptic-2} holds for $i=2$.

\item  $i\geq3$. From Lemma \ref{lem-2.1} (4), we see 
\begin{eqnarray*}
    F^{ii}\geq c_{2}\sum\limits_{k=1}^{n}F^{kk},
\end{eqnarray*}
and 
\begin{eqnarray}\label{sum-f}
    \sum\limits_{k=1}^{n}F^{kk}=\sum\limits_{k=1}^{n}(\sigma_{2;k}+\sigma_{1;k})=(n-2)\sigma_{2}+(n-1)\sigma_1,
\end{eqnarray}
From \eqref{2+1}, we derive that 
\begin{eqnarray}\label{case1-igeq3-Fii}
    F^{ii}\geq c\lambda_{1}.
\end{eqnarray}
Therefore, we have
\begin{eqnarray*}
F^{ii}|K_1+u_{ii}|^2\sim F^{ii}\geq  c\lambda_1,
\end{eqnarray*}
and \eqref{lambda-1} implies
\begin{eqnarray*}
F^{ii}=\sigma_{2;i}+\sigma_{1;i}\leq C\lambda_1,
\end{eqnarray*}
hence, we get  \eqref{uniform-elliptic-2} holds for $i\geq 3$.
\end{enumerate}

\

\textit{Case 2:} $\lambda_2> C_0$. In this case, by Lemma \ref{lem-2.1} (5), we have 
\begin{eqnarray}\label{case2-lambdai}
 |\lambda_{i}|\leq \max\{K, C_{3}\},\quad \forall ~3\leq i\leq n,
\end{eqnarray}
which implies
\begin{eqnarray}\label{case2-sigma2+sigma1}
\sigma_2+\sigma_1\sim\lambda_1\lambda_2,
\end{eqnarray}
and 
\begin{eqnarray}\label{case-s-control}
    |\sigma_{k;12}|\le  C,\quad \forall~1\leq k\leq 3.
    \end{eqnarray}
Next, we \textbf{claim} that 
\begin{eqnarray}\label{case2-claim}
-1.1\leq\sigma_{1;12}\leq-0.9.
\end{eqnarray}

If $\sigma_{1;12}\geq-0.9$, by \eqref{case2-claim} and \eqref{case-s-control}, we have
\begin{eqnarray*}
C\geq\sigma_3+\sigma_2&=&\lambda_1\lambda_2(1+\sigma_{1;12})+(\lambda_1+\lambda_2)(\sigma_{2;12}+\sigma_{1;12})+\sigma_{3;12}+\sigma_{2;12}\\
&\geq&0.1\lambda_1\lambda_2-C(\lambda_1+1),
\end{eqnarray*}
which is a contradiction provided that $C_0$ is sufficiently large.

If $\sigma_{1;12}\leq-1.1$,  using  \eqref{case2-lambdai} and \eqref{case-s-control} again, we get 
\begin{eqnarray*}
F^{11}=\sigma_{2;1}+\sigma_{1;1}= \lambda_2(1+\sigma_{1;12})+\sigma_{2;12}+\sigma_{1;12}\leq-0.1\lambda_2+C<0,
\end{eqnarray*}
we also derive a contradiction. Therefore, the \textbf{claim} \eqref{case2-claim} holds.

Next, we still prove that \eqref{elliptic} holds according the range of the index $i$.
\begin{enumerate}
\item $i=1$. Direct calculations yield
\begin{eqnarray*}
    F^{11}\lambda_{1}&=&\sigma_{3}+\s_{2}-\s_{3;1}-\sigma_{2;1}\\
    &=&f(x,u)-\lambda_2(\sigma_{2;12}+\sigma_{1;12})-(\sigma_{3;12}+\sigma_{2;12}),
\end{eqnarray*}
together this with \eqref{case2-lambdai}, \eqref{case-s-control}, \eqref{case2-claim} and \eqref{sigma212}, we conclude that 
\begin{eqnarray}\label{case2-i=1-lambda1F11}
\lambda_1F^{11}\geq c\lambda_2-C.
\end{eqnarray}
On the other hand, using \eqref{case2-claim} and \eqref{sigma212} again , we have
\begin{eqnarray}\label{case2-i=1}
F^{11}|K+u_{11}|^2\sim F^{11}\lambda_1^2&=&\lambda_1\left(
f(x,u)-\lambda_2(\sigma_{2;12}+\sigma_{1;12})-(\sigma_{3;12}+\sigma_{2;12})\right)\notag \\
&\leq&C(\lambda_{1}\lambda_{2}+\lambda_{1}).
\end{eqnarray}
Then \eqref{uniform-elliptic-2} for $i=1$ follows from \eqref{case2-sigma2+sigma1}, \eqref{case2-i=1-lambda1F11} and \eqref{case2-i=1}.

\item $i=2$. The similar argument as for case $i=1$, we obtain
\begin{eqnarray*}
F^{22}|K+u_{22}|^2\sim F^{22}\lambda_2^2=\lambda_2\left(\sigma_3+\sigma_2-\sigma_{3;2}-\sigma_{2;2}\right)\leq C\lambda_1\lambda_2+C\lambda_2,
\end{eqnarray*}

and
\begin{eqnarray}\label{case2-i=2-lambda2F22}
\lambda_2F^{22}=\sigma_2+\sigma_3-\sigma_{2;2}-\sigma_{3;2}\geq f(x,u)-\lambda_1\left(\sigma_{1;12}+\sigma_{2;12}\right)-C\geq c\lambda_1-C,
\end{eqnarray}
then  \eqref{uniform-elliptic-2} holds for $i=2$.

\item $i\geq3$. Combining Lemma \ref{lem-2.1} (4), \eqref{sum-f} and \eqref{case2-sigma2+sigma1}, we get 
\begin{eqnarray}\label{sum-Fii}
F^{ii}|K_1+u_{ii}|^2\sim F^{ii}\geq c\sum_{i=1}^nF^{ii}\geq c(\sigma_2+\sigma_1)\geq c\lambda_1\lambda_2,
\end{eqnarray}
and \eqref{case2-lambdai} implies
\begin{eqnarray*}
F^{ii}&=&\sigma_{2;i}+\sigma_{1;i}\\
&=&\lambda_1\lambda_2+\lambda_1(\sigma_{1;1i}+1)+\lambda_2(\sigma_{1;2i}+1)+\sigma_{2;12i}+\sigma_{1;12i}\leq C(\lambda_1\lambda_2+\lambda_1).
\end{eqnarray*}
Therefore, we derive that  \eqref{uniform-elliptic-2} holds for $i\geq3$.
\end{enumerate}

Combining the discussions in \textit{Case 1} and \textit{Case 2}, we conclude that \eqref{lem-uniform-elliptic} holds.
\end{proof}

\

Based on Lemma \ref{lem-uniform-elliptic}, we are now ready to establish the mean value type inequality for $b$.
\begin{lem}\label{lem-value}
Let $n\geq3$ and $f\in C^{1,1}(B_5\times\mathbb R\times \mathbb R^n)$ be a positive function. Let $u\in C^4(B_5)$ be an admissible, semi-convex solution of Eq. \eqref{main-eq}. Let $b=\log\lambda_{\max}$ and define 
\begin{eqnarray}\label{wide-b}
    \tilde{b}\coloneqq \log \{\lambda_{\max}(D^{2}u),J(n,K)\},
\end{eqnarray}
where $J(n, K)$ is the constant in Lemma \ref{lem-jacobi-ineq}. Then, we have
\begin{eqnarray}\label{weak-maximum-ineq}
\tilde{b}(0)\leq \int_{B_1}\tilde{b}(x)\det(D^2u(x)+K_1x)\,dx+C,
\end{eqnarray}
where $C$ is a constant depending only on $n$, $K$, $\|u\|_{C^{1}({B}_5)}$, $\min\limits_{{B}_5\times\mathbb R\times\mathbb R^n}f$ and $\|f\|_{C^{1,1}({B}_5 \times\mathbb R\times\mathbb R^n)}$.
\end{lem}

\begin{proof}
 Denote $b^{*}(y):=b(x(y))$.
Following arguments similar to those in  \cite[Proposition 2.3]{SY2020CVPDE} and \cite[Lemma 5.1]{L2025APDE} and combining Lemma \ref{lem-jacobi-ineq} and \ref{lem-uniform-elliptic},  we conclude that
\begin{eqnarray*}
G^{ij}b^{*}_{ij}&=&F^{ij}b_{ij}-\sum_k f_k b^*_k\\
&\geq&\varepsilon_0 F^{ij}b_i b_j-C(\Delta u+1)-\sum_k f_k b^*_k\\
&\geq&\varepsilon_0G^{ij}b^*_i b^*_j-C(\Delta u+1)-C|Db^*|~\geq~-C(\Delta u+1), 
\end{eqnarray*}
in the viscosity sense.  

Denote $\tilde{b}^*(y)=\tilde{b}(x(y))$, then
\begin{eqnarray*}
G^{ij}\tilde{b}^*_{ij}\geq-C(\Delta u+1)
\end{eqnarray*}
holds in the viscosity sense. Therefore, by Lemma \ref{lem-uniform-elliptic}, for sufficiently large $A>0$, we obtain
\begin{eqnarray*}
G^{ij}\left(\tilde{b}^*+A|y|^2\right)_{ij}\geq0,  
\end{eqnarray*}
 and
 \begin{eqnarray}\label{sub-equ}
     \tilde{G}^{ij}(\tilde{b}^{\star}+A|y|^{2}) \geq 0,
 \end{eqnarray}
 where $\tilde G^{ij}=\frac{G^{ij}}{\s_{2}(D^{2}u)+\sigma_{1}(D^{2}u)}$.

Using Lemma \ref{lem-uniform-elliptic} again, together with the local maximum principle (see \cite[Theorem 4.8]{CC1995Book}), we have
\begin{eqnarray*}
\tilde{b}^*(0)\leq C\int_{B_1^y}(\tilde{b}^*+A|y|^2)\,dy.  
\end{eqnarray*}
Since $D^{2}u \geq-KI$,  the map $y(x)=Du(x)+K_1x$ is uniformly monotone,  i.e.,  $$|y(x_1)-y(x_2)|\geq |x_1-x_2|.$$
Consequently, we have $x(B_1^y)\subset B_1$. Together with $y(0)=0$, we obtain
\begin{eqnarray*}
\tilde{b}(0)=\tilde{b}^*(0)&\leq& \int_{B_1^y}\tilde{b}^*(y)+A|y|^2\,dy \\
&\leq& \int_{B_1}\tilde{b}(x)\det(D^2u(x)+K_1x)\,dx+C.
\end{eqnarray*}
This completes the proof of Lemma \ref{lem-value}.
\end{proof}

\section{A key Lemma}\label{sec-5}
In this section, we prove a technical lemma, which will be used in dealing with the integral term on the right side of \eqref{weak-maximum-ineq}.
\begin{lem}\label{lem-crucial-ineq}
Let $n\geq3$ and $f\in C^{1,1}(B_5\times\mathbb R\times \mathbb R^n)$ be a positive function. Let $u\in C^4(B_5)$ be an admissible, semi-convex solution of Eq. \eqref{main-eq}. Let $x_0\in B_5$ and suppose that $D^2u$ is diagonalized at $x_0$. Then for any vector $\mathbf{a}=(a_1, \cdots, a_{n})\in \mathbb R^n$, we have
\begin{eqnarray}\label{key-ineu}
\sum_i|a_i\sigma_2^{ii}(D^2u)|\leq C\sum_iF^{ii}|a_i|^2+C\left(\sigma_2(D^2u)+\sigma_1(D^2u)\right), \quad 
\end{eqnarray}   
where $\sigma_{2}^{ij}=\frac{\partial \sigma_{2}(D^{2}u)}{\partial u_{ij}}$.
\end{lem}
\begin{proof}
First, by the Cauchy inequality, we have
\begin{eqnarray}
C F^{ii}(D^2u)|a_i|^2+C\left(\sigma_2(D^2u)+\sigma_1(D^2u)\right)&\geq& CF^{ii}(D^2u)^{1/2}\left(\sigma_2(D^2u)+\sigma_1(D^2u)\right)^{1/2}|a_i|. \quad \quad \quad  \label{cauchy-ineqy}
\end{eqnarray}
For any $1\leq i\leq n$, if the following inequality holds
\begin{eqnarray}\label{crucial-ineq-2}
|\sigma_2^{ii}(D^2u)|^2\leq CF^{ii}(D^2u)\left(\sigma_2(D^2u)+\sigma_1(D^2u)\right), \quad \forall 1\leq i\leq n,
\end{eqnarray}
together with \eqref{cauchy-ineqy}, we conclude that \eqref{key-ineu} is true. In the following, we turn our attention to proving the inequality \eqref{crucial-ineq-2}.

Assume that $\lambda_1$ is sufficiently large, otherwise, $\lambda(D^{2}u)$ lies in a compact set in $\tilde{\Gamma}_{3}$ and  \eqref{crucial-ineq-2} holds. Following the similar argument as  Lemma \ref{lem-uniform-elliptic}, we divide the proof into the following two cases. For simplicity, we  abbreviate $\sigma^{ii}_{2}(D^{2}u)$ as $\sigma_{1;i}$.

\

\textit{Case 1} $\lambda_2\leq C_0$ for some sufficiently large $C_0>0$. According to the range of the index $i$, we show \eqref{crucial-ineq-2} holds one by one.
\begin{enumerate}
\item $i=1$. By \eqref{lambda-1} and Lemma \ref{lem-2.1} (3), we obtain
\begin{eqnarray*}
\sigma_{1;1}\leq C \quad \text{and}\quad \lambda _{1}F^{11}\geq C.
\end{eqnarray*}
Together with \eqref{2+1}, we have \eqref{crucial-ineq-2} holds for $i=1$.

\item $i=2$. Since
\begin{eqnarray*}
    \s_{1;2}=\lambda_{1}+\sigma_{1;12}.
\end{eqnarray*}
Using \eqref{lambda-1} again, we get
\begin{eqnarray*}
\sigma_{1;2}\leq \lambda_1+C, 
\end{eqnarray*}
and \eqref{case1-i=2-F22} implies
\begin{eqnarray*}
 F^{22}\geq  c\lambda_1-C.
 \end{eqnarray*}
Therefore, we derive that \eqref{crucial-ineq-2} holds for $i=2$.

\item $i\geq3$. We have 
\begin{eqnarray*}
    \sigma_{1;i}=\lambda_{1}+\sigma_{1;1i}\leq \lambda_{1}+C,
\end{eqnarray*}
and  by \eqref{case1-igeq3-Fii}, we know that $F^{ii}\geq c\lambda_{1}$. Hence \eqref{crucial-ineq-2} holds for $i\geq 3$.
\end{enumerate}

\

\textit{Case 2} $\lambda_2> C_0$. 
We still prove \eqref{crucial-ineq-2} according to the range of the index $i$.
\begin{enumerate}
\item $i=1$. By  \eqref{case2-claim}, we derive that 
\begin{eqnarray}\label{2;1}
    \sigma_{1;1}=\lambda_{2}+\sigma_{1;12}\leq \lambda_{2}+C,
\end{eqnarray}
and  from \eqref{case2-i=1-lambda1F11}, we have
\begin{eqnarray}\label{f11}
F^{11}\geq c\lambda_1^{-1}\lambda_2-C\lambda_1^{-1}.
\end{eqnarray}
 Combining \eqref{case2-sigma2+sigma1}, \eqref{2;1} and \eqref{f11}, we derive that  \eqref{crucial-ineq-2} holds for $i=1$.

\item $i=2$. Using \eqref{case2-claim} again, we get 
\begin{eqnarray*}
    \sigma_{1;2}=\lambda_{1}+\sigma_{1;12}\leq \lambda_{1}+C,
\end{eqnarray*}
and \eqref{case2-i=2-lambda2F22} implies
\begin{eqnarray*}
F^{22}\geq c\lambda_1\lambda_2^{-1}-C\lambda_2^{-1}.
\end{eqnarray*}
Using \eqref{case2-sigma2+sigma1},  we can conclude that \eqref{crucial-ineq-2} holds for $i=2$.

\item $i\geq3$. From \eqref{case-s-control}, we have
\begin{eqnarray*}
    \sigma_{1;i}=\lambda_{1}+\lambda_{2}+\s_{1;12i}\leq \lambda_{1}+\lambda_{2}+C,
\end{eqnarray*}
and by \eqref{sum-Fii}, we know that $F^{ii}\geq c\lambda_{1}\lambda_{2}$. Therefore, we have \eqref{crucial-ineq-2} holds for $i\geq3$.
\end{enumerate}

Finally, combining the discussion in \textit{Case 1} and \textit{Case 2}, we conclude that \eqref{crucial-ineq-2} holds.

\end{proof}

\section{Proof of Theorem \ref{thm-quotient} and Theorem \ref{thm-main}}\label{sec-6}

Building on the preparatory work presented earlier, we proceed to complete the proof of Theorem \ref{thm-main} and Theorem \ref{thm-quotient}.
\begin{proof}[\textbf{Proof of Theorem \ref{thm-main}}]
Assume that $\lambda_1$ is sufficiently large at origin,  otherwise, \eqref{0-est} holds automatically. 

By Lemma \ref{lem-2.1} (5), we obtain
\begin{eqnarray}\label{det-control}
    \det(D^{2}u+K_{1}I)\leq  C(\lambda_{1}+\lambda_{1}\lambda_{2}+1)\leq C(\s_{2}(D^{2}u)+\Delta u+1).
\end{eqnarray}
Substituting \eqref{det-control} into \eqref{weak-maximum-ineq}, we get
\begin{eqnarray*}
\tilde{b}(0)=\tilde{b}^*(0)&\leq& C\int_{B_1}\tilde{b}(x)\sigma_2(D^2u)\,dx+C\int_{B_1}\tilde{b}(x)\Delta u\,dx+C\int_{B_1}\tilde{b}(x)\,dx\\
&\coloneqq &I_{1}+{I}_{2}+{I}_{3}.
\end{eqnarray*}
Next, we handle the above three terms one by one.  

Let $\varphi(x)$ be a smooth cut-off function such that $\varphi\equiv1$ in $B_1$ and vanishes outside $B_2$. For the term $I_{1}$,  we have
\begin{eqnarray*}
{I}_{1}&\leq& C\int_{B_2}\varphi^3(x)\tilde{b}(x)\sigma_2^{ij}u_{ij}\,dx\\
&=& -C\int_{B_2}\varphi^2(x)\varphi_j(x)\tilde{b}(x)\sigma_2^{ij}u_j\,dx-C\int_{B_2}\varphi^3(x)\tilde{b}_j(x)\sigma_2^{ij}u_i dx\\
&&-\int_{B{2}}\varphi^{3}(x)\tilde{b}(x)(\sigma_{2}^{ij})_{j}u_{i}dx\\
&=&-C\int_{B_2}\varphi^2(x)\varphi_j(x)\tilde{b}(x)\sigma_2^{ij}u_j\,dx-C\int_{B_2}\varphi^3(x)\tilde{b}_j(x)\sigma_2^{ij}u_i\,dx.\\
&\coloneqq&I_{11}+{I_{12}}.
\end{eqnarray*}
In the second last equality, we used the fact that $\sum\limits_{j=1}^{n}(\sigma_{2}^{ij})_{j}=0$. In the following, we analyze the above two terms carefully.

For the term $I_{11}$, we  get
\begin{eqnarray}\label{i11}
{I}_{11}&=&-C\int_{B_2}\varphi^2(x)\varphi_i(x)\tilde{b}(x)\sigma_2^{ij}u_j\,dx\leq C\int_{B_2}\varphi^2(x)\tilde{b}(x)\Delta u\,dx \notag   \\
&\leq &  C\int_{B_2}\tilde{b}(x)\,dx+C\int_{B_2}\varphi^2(x)|\n \tilde{b}|\,dx.
\end{eqnarray}
From the definition of $\tilde b$, we have $\tilde b (x)\leq C(\Delta u+1)$, it follows that 
\begin{eqnarray} \label{i1-control}
    \int_{B_{2}}\tilde b(x)dx\leq C\int_{B_{2}}(\Delta u+1)dx\leq C.
\end{eqnarray}
For a fixed point $x\in B_{2}$, we choose an orthogonal frame such that $D^{2} u(x)$ is diagonal and $F^{ij}$ is also diagonal at $x$. 
For any $1\leq i\leq n$, by the Cauchy inequality, we have
\begin{eqnarray}\label{bi-cauchy}
    2|\tilde b_{i}|\leq  F^{ii}|\tilde b_{i}|^{2}+\frac{1}{F^{ii}},
\end{eqnarray}
when $1\leq i\leq 2$, by Lemma \ref{lem-2.1} (3), we get 
\begin{eqnarray*}
    \frac{1}{F^{ii}}\leq \frac{\lambda_{i}}{c_{1}F(\lambda)}\leq C(\Delta u+1),
\end{eqnarray*}
when $3\leq i\leq n$, by  \cite[Lemma 1.5]{spruck} and the concavity of $\log F(\lambda)$, we conclude that $F^{ii}\geq F^{11}$, then 
\begin{eqnarray}\label{f33}
    \frac{1}{F^{ii}}\leq \frac{1}{F^{11}}\leq C(\Delta u+1).
\end{eqnarray}
Combining \eqref{bi-cauchy}, \eqref{f11} and \eqref{f33}, we conclude that 
\begin{eqnarray*}
    |\n \tilde b|\leq F^{ij} \tilde b_{i} \tilde b_{j}+C(\Delta u+1),
\end{eqnarray*}
and 
\begin{eqnarray}\label{i11-1}
    \int_{B_{2}}\varphi^{2}(x)|\n \tilde b|dx\leq C\int_{B_{2}}\varphi^{2}F^{ij} \tilde b_{i}\tilde b_{j}dx+C.
\end{eqnarray}
Using the integral Jacobi inequality \eqref{inter-jacobi}, we have
\begin{eqnarray}\label{i11-2}
\int_{B_{2}}\varphi^{2} F^{ij}\tilde {b}_{i}\tilde b_{j}dx&\leq& -C\int_{B_2}\varphi(x)\varphi_j(x)F^{ij}\tilde{b}_i\,dx+C \notag \\
&\leq& \frac{1}{2}\int_{B_2}\varphi^2(x)F^{ii}|\tilde{b}_i|^2\,dx+C\left(\int_{B_2}F^{ij}\varphi_i\varphi_j\,dx+1\right),
\end{eqnarray}
in the first inequality, we used the fact that $\sum\limits_{j=1}^{n}(F^{ij})_{j}=0$.  On the other hand, by \eqref{sum-f}, we get 
\begin{eqnarray*}
    \int_{B_{2}}F^{ij}\varphi_{i}\varphi_{j}dx&\leq &C\int_{B_{2}}\sum\limits_{k=1}^{n}F^{kk}dx
    \leq C\int_{B_{2}}(\sigma_{2}(D^{2}u)+\Delta u)dx \notag \\&\leq& C\int_{B_3}\psi(x)\sigma_2^{ij}u_{ij}\,dx+C
\leq C\int_{B_3}\Delta u\,dx+C\leq C,
\end{eqnarray*}
where $\psi(x)$ is a smooth cut-off function such that $\psi\equiv1$ in $B_2$ and vanishes outside $B_3$.

Together with \eqref{i11-1} and \eqref{i11-2}, we derive that 
\begin{eqnarray}
    \int_{B_{2}}\varphi^{2}|\n \tilde b|dx\leq C. \label{i11-3}
\end{eqnarray}
Substituting \eqref{i11-3} and \eqref{i1-control} into \eqref{i11}, we have \begin{eqnarray}\label{i11-est}
    I_{11}\leq C.
\end{eqnarray}

For the term ${I_{12}}$, by lemma \ref{lem-crucial-ineq}, we have
\begin{eqnarray}\label{i22}
I_{12}&\leq& C\int_{B_2}\varphi^3(x)F^{ij}\tilde{b}_i\tilde{b}_{j}dx+C\int_{B_2}\varphi^3(x)\left(\sigma_2(D^2u)+\Delta u\right)\,dx\leq C.
\end{eqnarray}
Combining \eqref{i11-est} and \eqref{i22}, we obtain
\begin{eqnarray*}
I_{1}=I_{11}+I_{12}\leq C.
\end{eqnarray*}
A similar argument for $I_{2}$ and $I_{3}$, we have $I_{2}+I_{3}\leq C$ holds.

Therefore, we conclude that 
\begin{eqnarray*}
    \tilde{b}(0)\leq I_{1}+I_{2}+I_{3} \leq C,
\end{eqnarray*}
which completes the proof of Theorem \ref{thm-main}.
\end{proof}

\

\begin{proof}[\textbf{Proof of Theorem \ref{thm-quotient}}]
According to the value of $l$, we divide the proof of Theorem \ref{thm-quotient} into the following two cases.

\

\textit{Case~1} $l=1$. In view of rescaling, we may assume
\begin{eqnarray}\label{u-equ-3}
\frac{\sigma_3(D^2u)}{\sigma_1(D^2u)}=C_0,
\end{eqnarray}
where the positive constant $C_0=\sqrt[3]{\frac{9(n-1)(n-2)}{8n^2}}$.

Let
\begin{eqnarray*}
v(x)\coloneqq u(x)-\frac{a}{2}|x|^2, 
\end{eqnarray*}
where the positive constant $a=\sqrt[3]{\frac{3}{n(n-1)(n-2)}}$.  Direct calculations yield
\begin{eqnarray*}
0&=&\sigma_3(D^2u)-C_0\sigma_1(D^2u)~=~\sigma_3(D^2v+aI)-C_0\sigma_1(D^2v+aI)\\
&=&\sigma_3(D^2v)+(n-2)a\sigma_2(D^2v)+\frac{(n-1)(n-2)}{2}a^2\sigma_1(D^2v)\\
&&+\frac{n(n-1)(n-2)}{6}a^3-C_0\sigma_1(D^2v)-C_0na\\
&=&\sigma_3(D^2v)+(n-2)a\sigma_2(D^2v)-1,
\end{eqnarray*}
where we have used the fact that $C_0=\frac{(n-1)(n-2)}{2}a^2$. Therefore, $v$ satisfies the equation 
\begin{eqnarray}\label{eq-quotient-v}
\sigma_3(D^2v)+(n-2)a\sigma_2(D^2v)=1.
\end{eqnarray}
In view of remark \ref{rmk-quotient}, we only need to show that $v$ is semi-convex and $\lambda(D^{2}v)\in \tilde \Gamma_{3}$, then by Theorem \ref{thm-main}, we derive that \eqref{0-est-quotient} holds.

First, the semi-convexity of $v$ is a direct consequence of the semi-convexity of $u$. Next, we prove that $\lambda(D^{2}v)\in \tilde{\Gamma}_{3}$. By \cite[Lemma 2.3]{LR2023JFA} and Eq. \eqref{eq-quotient-v}, it suffices to check that $v\in\tilde\Gamma_2$, i.e., $v$ satisfies
\begin{eqnarray}\label{quotient-eq-sigma1v}
\sigma_1(D^2v)>0
\end{eqnarray}
and 
\begin{eqnarray}\label{quotient-eq-sigma2v+sigma1v}
\sigma_{2}(D^2v)+(n-2)a\sigma_1(D^2v)>0.
\end{eqnarray}
By the Newton-Maclaurin inequality and Eq. \eqref{u-equ-3}, we have
\begin{eqnarray*}
\left(\sigma_1(D^2u)\right)^2\geq\frac{6n^2}{(n-1)(n-2)}\frac{\sigma_3(D^2u)}{\sigma_1(D^2u)}= 3n^2a^2,   
\end{eqnarray*}
it follows that
\begin{eqnarray*}
\sigma_1(D^2v)=\sigma_1(D^2u)-na>0.    
\end{eqnarray*}
Hence, \eqref{quotient-eq-sigma1v} holds. 

For \eqref{quotient-eq-sigma2v+sigma1v}, using the Newton-Maclaurin inequality again, we get
\begin{eqnarray*}
\sigma_2(D^2u)^2\geq \frac{3(n-1)}{2(n-2)}\sigma_3(D^2u)\cdot\sigma_1(D^2u)=\frac{3(n-1)^2}{4}a^2\sigma_1(D^2u)^2,    
\end{eqnarray*}
and
\begin{eqnarray*}
\sigma_2(D^2v)+(n-2)a\sigma_1(D^2v)&=&\sigma_2(D^2u)-(n-1)a\sigma_1(D^2u)+\frac{n(n-1)}{2}a^2\\
&&+(n-2)a\sigma_1(D^2u)-n(n-2)a^2\\
&=&\sigma_2(D^2u)-a\sigma_1(D^2u)-\frac{n(n-3)}{2}a^2.\\
&\geq& \left(\frac{\sqrt{3}}{2}(n-1)-1\right)a\sigma_1(D^2u)-\frac{n(n-3)}{2}a^2>0.
\end{eqnarray*}
Therefore, we conclude that \eqref{quotient-eq-sigma2v+sigma1v} holds, and we complete the proof of Theorem \ref{thm-quotient} for $k=1$.

\

\textit{Case~2} $l=2$. Similar to before, we may assume
\begin{eqnarray}\label{u-equ-31}
\frac{\sigma_3(D^2u)}{\sigma_2(D^2u)}=C_0,
\end{eqnarray}
where the positive constant $C_0=\sqrt[3]{\frac{3(n-2)^2}{2n(n-1)}}$.

Let
\begin{eqnarray*}
v(x):=u(x)-\frac{a}{2}|x|^2 
\end{eqnarray*}
where the positive constant $a=\sqrt[3]{\frac{12}{n(n-1)(n-2)}}$. Then, we have $$C_{0}=\frac{(n-2)a}{2}.$$
Direct calculations yield 
\begin{eqnarray*}
0&=&\sigma_3(D^2u)-C_0\sigma_2(D^2u)~=~\sigma_3(D^2v+aI)-C_0\sigma_2(D^2v+aI)\\
&=&\sigma_3(D^2v)+(n-2)a\sigma_2(D^2v)+\frac{(n-1)(n-2)}{2}a^2\sigma_1(D^2v)+\frac{n(n-1)(n-2)}{6}a^3\\
&&-C_0\sigma_2(D^2v)-C_0(n-1)a\sigma_1(D^2v)-C_0\frac{n(n-1)}{2}a^2\\
&=&\sigma_3(D^2v)+\frac{n-2}{2}a\sigma_2(D^2v)-1,
\end{eqnarray*}
Therefore, $v$ satisfies the equation 
\begin{eqnarray}\label{equ-v2}
\sigma_3(D^2v)+\frac{n-2}{2}a\sigma_2(D^2v)=1.
\end{eqnarray}
By a similar argument as in the proof of Theorem \ref{thm-quotient}, we should show that $v$ satisfies \eqref{quotient-eq-sigma1v} and \eqref{quotient-eq-sigma2v+sigma1v}. 
For \eqref{quotient-eq-sigma1v}, by the Newton-Maclaurin inequality and Eq. \eqref{u-equ-31}, we have
\begin{eqnarray*}
\sigma_1(D^2u)\geq \frac{3n}{n-2}\frac{\sigma_3(D^2u)}{\sigma_2(D^2u)}=\frac{3n}{2}a,
\end{eqnarray*}
it follows that
\begin{eqnarray*}
\sigma_1(D^2v)=\sigma_1(D^2u)-na>0.    
\end{eqnarray*}
For \eqref{quotient-eq-sigma2v+sigma1v}, we obtain 
\begin{eqnarray*}
\sigma_2(D^2u)\geq \frac{3(n-1)}{2(n-2)}\frac{\sigma_3(D^2u)\cdot\sigma_1(D^2u)}{\sigma_2(D^2u)}=\frac{3(n-1)}{4}a\sigma_1(D^{2}u),
\end{eqnarray*}
and 
\begin{eqnarray*}
\sigma_2(D^2v)+\frac{n-2}{2}a\sigma_1(D^2v)&=&\sigma_2(D^2u)-(n-1)a\sigma_1(D^2u)+\frac{n(n-1)}{2}a^2\\
&&+\frac{n-2}{2}a\sigma_1(D^2u)-\frac{n(n-2)}{2}a^2\\
&=&\sigma_2(D^2u)-\frac{n}{2}a\sigma_1(D^2u)+\frac{n}{2}a^2.\\
&\geq& \frac{n-3}{4}a\sigma_1(D^2u)+\frac{n}{2}a^2>0.
\end{eqnarray*}
Therefore, we derive that \eqref{quotient-eq-sigma2v+sigma1v} holds, and we complete the proof of Theorem \ref{thm-quotient} for $k=2$.
\end{proof}

\section{A singular solution}\label{sec-7}
In this section, when $k\geq 4$, we construct a singular solution for Eq. \eqref{sum-hessian} by extending Pogorelov's example \cite{pogo}. 

\begin{proof}[\textbf{Proof of Theorem \ref{thm-counter example}}]
When $k\geq 4$, we consider the function  
\begin{eqnarray*}
    u^{\s}(x)\coloneqq (1+x_{1}^{2})(\s+\sum_{i=2}^{k-1}x_{i}^{2})^{\frac{\a}{2}},
\end{eqnarray*}
where $\a=2-\frac{2}{k-1}$. Since $k\geq 4$, we have $1<\a<2$. Direct calculations yield
\begin{eqnarray*}
    u_{11}^{\s}=2\left(\s+\sum\limits_{i=1}^{k-1}x_{i}^{2}\right)^{\frac{\a}{2}},  \quad \quad  u^{\s}_{1i}=\left(\s+\sum\limits_{i=2}^{k-1}x_{i}^{2}\right)^{\frac{\a}{2}-1} 2\a x_{1}x_{i}, \quad 1\leq i\leq k-1.
\end{eqnarray*}
For any $2\leq i\leq k-2$, we have
\begin{eqnarray*}
    u_{ii}^{\s}=(1+x_{1}^{2})\left(\s+\sum\limits_{i=2}^{k-1}x_{i}^{2}\right)^{\frac{\a}{2}-2}\left( \a(\a-2)x_{i}^{2}+\a(\s+\sum\limits_{i=2}^{k-1}x_{i}^{2})
    \right),
\end{eqnarray*}
and 
\begin{eqnarray*}
    u_{ij}^{\s}=(1+x_{1}^{2})\left(\s+\sum\limits_{i=2}^{k-1}x_{i}^{2}\right)^{\frac{\a}{2}-2}\a(\a-2)x_{j}x_{i}, \quad \text{for}~2\leq j\leq k-1, i\neq j. 
\end{eqnarray*}
Moreover, we have
\begin{eqnarray*}
    u_{ij}=0, \quad \text{for}~ k\leq i, j\leq n. 
\end{eqnarray*}
Then there exists a $r>0$ such that $x\in B_{r}$, we have 
\begin{eqnarray*}
    \sigma_{k-1}(D^{2}u^{\s})=\left(\s+\sum\limits_{i=2}^{k-1}x_{i}^{2}\right)^{\frac{(k-1)\a}{2}-(k-2)}f^{\s}=f^{\s},
\end{eqnarray*}
and 
\begin{eqnarray*}
    \sigma_{k}(D^{\s}u)=0,
\end{eqnarray*}
where $f^{\s}$ is a smooth positive function with positive lower bound in $B_{r}$. Therefore, we obtain
\begin{eqnarray*}
    \sigma_{k}(D^{2}u^{\s})+\sigma_{k-1}(D^{2}u^{\s})=f^{\s}.
\end{eqnarray*}
Let $\s\rightarrow 0$, we have $u^{\s}\rightarrow u$ and $f^{\s}\rightarrow f$ locally uniformly, where $u=(1+x_{1}^{2})\left(\sum\limits_{i=2}^{k-1}x_{i}^{2}\right)^{\frac{\a}{2}}$
 and $f$ is a smooth positive function with a positive lower bound in $B_{r}$, and $u$ is the viscosity solution of 
 \begin{eqnarray*}
     \sigma_{k}(D^{2}u)+\sigma_{k-1}(D^{2}u)=f.
 \end{eqnarray*}
We can check that $u$ is convex and Lipschitz  continuous and $u\notin C^{1,\beta}(B_r)$  for some $0<\beta <1$. Hence, $u$ is the desired singular solution and we complete the proof of Theorem \ref{thm-counter example}.
\end{proof}

\section{Proof of Theorem \ref{thm-quotient-rigidity} and Theorem \ref{thm-rigidity}}
In this section,  we establish several rigidity theorems. The proof of Theorem \ref{thm-quotient-rigidity} is the same as the proof of Theorem \ref{thm-rigidity} (1), so we omit it here. Next, we complete the proof of Theorem \ref{thm-rigidity}.

\begin{proof}[\textbf{Proof of Theorem \ref{thm-rigidity}}]
Denote $$u_R(x)\coloneqq\frac{u(Rx)}{R^2},$$
for some $R>0$. Then $u_{R}(x)$ satisfies the equation 
\begin{eqnarray*}
    \sigma_{3}(D^{2}u_{R})+\sigma_{2}(D^{2}u_{R})=c_{0}.
\end{eqnarray*}
Under condition (1),  by Theorem \ref{thm-main} and the Evans-Krylov-Safonov theory, we get
\begin{eqnarray*}
[D^2u]_{C^\alpha(B_R)}=\frac{[D^2u_R]_{C^\alpha(B_1)}}{R^\alpha}&\leq& \frac{C\|D^2u_R\|_{L^{\infty}(B_{2})}}{R^\alpha}\\
&\leq &\frac{C\|u_R\|_{L^\infty(B_{3})}}{R^{\alpha}} =\frac{C\|u\|_{L^{\infty}(B_{3R})}}{R^{2+\alpha}}\to 0
\end{eqnarray*}
as $R\to\infty$, where $\alpha=\alpha(n)>0$. We conclude that $u$ must be quadratic.

Under condition (2), we modify the argument in \cite{CY2010DCDS} to complete the proof. Let $w(y)$ be the Legendre transform of $\tilde{u}(x)\coloneqq u(x)+\frac{K}{2}|x|^2$ with $K>\frac{1}{n-2}-\delta$. Define
$$
G\left(D^2w(y)\right)\coloneqq -F\left(-KI+(D^2w(y))^{-1}\right)=-F\left(D^2u(x)\right)=-c_0.
$$
Let $\mu=(\mu_{1}, \cdots, \mu_{n})$ be the eigenvalues of $D^2w(y)$, which are  positive and bounded. Direct calculation yield 
\begin{eqnarray*}
0&=&G(D^2w)+c_0~=~-F\left(-KI+(D^2w(y))^{-1}\right)+c_0\\
&=&-\sigma_n(\mu)^{-1}\Biggl(\sigma_{n-3}(\mu)+\left(1-(n-2)K\right)\sigma_{n-2}(\mu)+\Bigl(\frac{(n-1)(n-2)}{2}K^2\\
&&-(n-1)K\Bigr)\sigma_{n-1}(\mu)+\left(-\frac{n(n-1)(n-2)}{6}K^3+\frac{n(n-1)}{2}K^2-c_0\right)\sigma_n(\mu)\Biggr).
\end{eqnarray*}
As shown in Lemma \ref{lem-uniform-elliptic}, $G^{ij}$ is uniformly elliptic by a suitable normalization $f(\mu)$. Therefore, denote 
\begin{eqnarray*}
\Sigma\coloneqq\{\mu\in\mathbb R^n \mid 0=G(\mu)+c_0\},
\end{eqnarray*}
and $H^{ij}\coloneqq f(\mu(D^2w))G^{ij}$. The linearized operator $H^{ij}\partial_{ij}$ of the equation
\begin{eqnarray*}
0&=&H(D^2w)\coloneqq f(\mu)\left(G(D^2w)+c_0\right)\\
&=&-\sigma_n(\mu)^{-1}f(\mu)\Biggl(\sigma_{n-3}(\mu)+\left(1-(n-2)K\right)\sigma_{n-2}(\mu)+\Bigl(\frac{(n-1)(n-2)}{2}K^2\\
&&-(n-1)K\Bigr)\sigma_{n-1}(\mu)+\left(-\frac{n(n-1)(n-2)}{6}K^3+\frac{n(n-1)}{2}K^2-c_0\right)\sigma_n(\mu)\Biggr)
\end{eqnarray*}
is uniformly elliptic on $\Sigma$.

Next, we aim to show that $\Sigma$ is a convex hypersurface in $\mathbb R^n$. Since 
$$
c_0> -\frac{n(n-1)(n-2)}{6}\left(\frac{1}{n-2}-\delta\right)^3+\frac{n(n-1)}{2}\left(\frac{1}{n-2}-\delta\right)^2,
$$
we can choose $\frac{1}{n-2}\geq K>\frac{1}{n-2}-\delta$ such that
$$
-\frac{n(n-1)(n-2)}{6}K^3+\frac{n(n-1)}{2}K^2-c_0<0.
$$
Therefore, $\Sigma$ can be viewed as the zero level set of
$$
g(\mu)\coloneqq A_n\sigma_n(\mu)+A_{n-1}\sigma_{n-1}(\mu)+A_{n-2}\sigma_{n-2}(\mu)+\sigma_{n-3}(\mu),
$$
where $A_n<0$  and $A_{n-2}\geq0$.

By \cite[Proposition 2.2]{GZ2021PAMQ}, we know $-\frac{g(\mu)}{\sigma_{n-1}(\mu)}$ is elliptic and concave. It follows that $\Sigma$ is a convex hypersurface.

In summary, we have proved that $w$ satisfies a uniformly elliptic equation on a convex hypersurface $\Sigma$, with $D^2w$ bounded. By the Evans-Krylov-Safonov theory, we obtain
$$
[D^2w]_{C^{\alpha}(B_R)}\leq \frac{C\|D^2w\|_{L^\infty(B_{2R})}}{R^{\alpha}}\leq \frac{C}{R^{\alpha}} \to 0 \quad \text{as}\quad R\to\infty.
$$
We conclude that $w$ is a quadratic polynomial, so is $u$.

Under condition (3), let $w(y)$ be the Legendre transform of $\tilde{u}(x)\coloneqq u(x)+\frac{K_0}{2}|x|^2$ for the unique $K_0\in(0,\frac{2}{n-2}]$ such that $-\frac{n(n-1)(n-2)}{6}K_0^3+\frac{n(n-1)}{2}K_0^2=c_0$.

The same notation and arguments as in condition (2) (with $K=K_0$) yield that the linearized operator $H^{ij}\partial_{ij}$ of the equation
\begin{eqnarray*}
0&=&H(D^2w)\coloneqq f(\mu)\left(G(D^2w)+c_0\right)\\
&=&-\sigma_n(\mu)^{-1}f(\mu)\left(\sigma_{n-3}(\mu)+A_{n-2}\sigma_{n-2}(\mu)+A_{n-1}\sigma_{n-1}(\mu)\right)\\
&\coloneqq&-\sigma_n(\mu)^{-1}f(\mu)g(\mu)
\end{eqnarray*}
is uniformly elliptic on $\Sigma$.

Moreover, since $K_0\in(0,\frac{2}{n-2}]$, we know $A_{n-1}\leq 0$. Therefore, by \cite[Proposition 2.2]{GZ2021PAMQ} again, we know $-\frac{g(\mu)}{\sigma_{n-2}(\mu)}$ is elliptic and concave. It follows that $\Sigma$ is a convex hypersurface. 

By the same arguments as in condition (2), we obtain that $u$ is quadratic.

\end{proof}

\bigskip

\bigskip

\noindent\textit{Acknowledgment:}
 The authors would like to express sincere gratitude to Prof. Xinan Ma for his ongoing support and encouragement.  X.M. was supported by the National Key R $\&$ D Program of China (No. 2020YFA0712800) and the Postdoctoral Fellowship Program of CPSF under Grant Numbers 2025T180843 and 2025M773082.

\bigskip

\printbibliography

\end{document}